\documentclass[11pt]{amsart}

\usepackage{amsthm}                   
\usepackage{amsmath}                  
\usepackage{amsfonts}                 
\usepackage{amssymb}                  
\usepackage{mathrsfs}                 
\usepackage{bbm}                      
\usepackage[margin=1.25in]{geometry}  
\usepackage{multido}                  
\usepackage{pstricks}                 
\usepackage{pst-plot}                 
\usepackage{pst-3dplot}               
\usepackage{enumerate}                
\usepackage{mathtools}								
\usepackage{url}

\theoremstyle{definition}
\newtheorem{defi}{Definition}[section]

\newtheorem{ex}[defi]{Example}

\theoremstyle{plain}
\newtheorem{thm}[defi]{Theorem}
\newtheorem{lemma}[defi]{Lemma}
\newtheorem{cor}[defi]{Corollary}
\newtheorem{prop}[defi]{Proposition}

\numberwithin{equation}{section}

\newcommand{\R}{\ensuremath{\mathbbm{R}}}     
\newcommand{\Z}{\ensuremath{\mathbbm{Z}}}     
\newcommand{\Q}{\ensuremath{\mathbbm{Q}}}     

\def\ts{\textstyle}

\def\lt{\left}
\def\rt{\right}

\newcommand{\Rd}{\ensuremath{\R^d}}                       
\newcommand{\set}[1]{\left\{#1\right\}}                   
\newcommand{\gen}[1]{\lt\langle #1\rt\rangle}             

\makeatletter
\def\imod#1{\allowbreak\mkern10mu({\operator@font mod}\,\,#1)} 
\def\@setcopyright{}                                           
\def\serieslogo@{}
\makeatother

\newcommand{\G}{\ensuremath{\Gamma}}        
\renewcommand{\S}{\ensuremath{\Sigma}}      
\newcommand{\vep}{\ensuremath{\varepsilon}} 

\begin{document}
	\author[J.C.H.~Arias]{Jeanine Concepcion H.~Arias}
	\email[J.C.H.~Arias]{jcarias@math.upd.edu.ph}
	
	\author[E.D.~Gabinete]{Evelyn D.~Gabinete}
	\email[E.D.~Gabinete]{evelyn.dumag@gmail.com}
	
	\author[M.J.C.~Loquias]{Manuel Joseph C.~Loquias}
	\email[M.J.C.~Loquias]{mjcloquias@math.upd.edu.ph}	

	\address{Institute of Mathematics, University of the Philippines Diliman, 1101 Quezon City, Philippines}
	
	\title[Coincidences of a shifted hexagonal lattice and the hexagonal packing]{Coincidences of a shifted hexagonal lattice\\[6pt] and the hexagonal packing}

	\begin{abstract}
	  A geometric study of twin and grain boundaries in crystals and quasicrystals is achieved via coincidence site lattices (CSLs) and coincidence site modules (CSMs), respectively.  	
		Recently, coincidences of shifted lattices and multilattices (i.e. finite unions of shifted copies of a lattice) have been investigated.  Here, we solve the coincidence 
		problem for a shifted hexagonal lattice.  This result allows us to analyze the coincidence isometries of the hexagonal packing by viewing the hexagonal packing as a multilattice.
	\end{abstract}

	

	\date{\today}

	\maketitle

	\section{Introduction and Preliminaries}
		Coincidence site lattices (CSLs) are used in describing and classifying grain boundaries, twins, and interfaces of crystals (see for instance \cite{B70,GBW74}), and more
		recently, of quasicrystals (see \cite{W93}).  A general and mathematical treatment of the coincidence problem for lattices and $\Z$-modules can be found in~\cite{B97}.  
		In~\cite{LZ10,LZxx}, the idea of coincidences was extended to include affine isometries, shifted lattices, and sets of points formed by the union of a lattice with a finite 
		number of shifted copies of the lattice (referred to as multilattices, see~\cite{PZ98} and references therein). This paper is a continuation of the above-mentioned works where the 
		coincidence problem for a shifted hexagonal lattice and the hexagonal packing are solved.

		A discrete set $\G$ in $\R^d$ is a \emph{lattice} if it is the $\Z$-span of $d$~vectors that are linearly independent.  An $R\in O(d)$ is said to be a (\emph{linear}) 
		\emph{coincidence isometry of $\G$} if $\G\cap R\G$ is a sublattice of full rank in $\G$.  The intersection $\G\cap R\G$ is called the \emph{coincidence site lattice} (CSL) of $\G$
		generated by $R$ and is denoted by $\G(R)$.  The \emph{coincidence index of $R$ with respect to $\G$} is the group index of $\G(R)$ in $\G$, written $\S_{\G}(R)$.  
		Geometrically, $\S_{\G}(R)$ gives the volume of a fundamental domain of $\G(R)$ per volume of a fundamental domain of $\G$.  It is known that the sets $OC(\G)$ and $SOC(\G)$ of
		coincidence isometries and coincidence rotations of $\G$, respectively, form groups.  Background on the mathematical formulation of CSLs can be found in~\cite{B97}.
		
		These notions were generalized to shifted lattices and multilattices in \cite{LZ10,LZxx}.  Given a lattice $\G$ and a vector $x$ in~$\Rd$, the shifted lattice $x+\G$ is obtained by
		translating every point of $\G$ by $x$.  By a \emph{cosublattice of the shifted lattice $x+\G$}, we mean a subset $x+(\ell+\G')$ of $x+\G$, where $\ell+\G'$ is a coset of some
		sublattice $\G'$ of $\G$.  A \emph{linear coincidence isometry $R$ of $x+\G$} is an $R\in O(d)$ such that $(x+\G)\cap R(x+\G)$ contains a cosublattice of $x+\G$.  The \emph{CSL of 
		$x+\G$} obtained from $R$ and the \emph{coincidence index $\S_{x+\G}(R)$ of $R$ with respect to $x+\G$} are defined analogously.  The sets of linear coincidence isometries and
		coincidence rotations of $x+\G$ are denoted by $OC(x+\G)$ and $SOC(x+\G)$, respectively.  The following result from~\cite{LZ10,LZxx} characterizes $OC(x+\G)$ and identifies the CSL
		and coincidence index of an $R\in OC(x+\G)$.
		\begin{thm}\label{oc}
			Let $\G$ be a lattice in $\R^d$ and $x\in\R^d$. 
			\begin{enumerate}[i.]
				\item Then $OC(x+\G)=\set{R\in OC(\G) \mid Rx-x\in\G+R\G}$ and it is a group whenever it is closed under composition.
				
				\item If $R\in OC(x+\G)$ with $Rx-x\in\ell+R\G$ for some $\ell\in\G$, then \[(x+\G)\cap {R(x+\G)}=(x+\ell)+\G(R)\] and $\S_{x+\G}(R)=\S_{\G}(R)$.
			\end{enumerate}
		\end{thm}
		
		It is also known that corresponding shifted lattices derived from lattices of the same Bravais type share the same set (up to conjugation) of linear coincidence isometries. 
		Furthermore, it is enough to look at values of $x$ in a fundamental domain of the symmetry group of $\G$ to determine all possible sets $OC(x+\G)$.

		A \emph{multilattice} $L$ is the union of a finite number of shifted copies of a lattice~$\G$, i.e., $L=\bigcup_{k=0}^{m-1}(x_k+\G)$, where $x_0=0$ and $x_{k_1}-x_{k_2}
		\in\R^d\setminus\G$ whenever $k_1 \neq k_2$.  In general, a multilattice need not be a lattice.  A linear isometry $R$ is a (\emph{linear}) \emph{coincidence isometry of $L$} if 
		$L\cap RL$ contains a cosublattice of some shifted lattice $x_k+\G$ in $L$.  Here, the intersection $L(R)\vcentcolon=L\cap RL$ is referred to as the \emph{coincidence site
		multilattice} (\emph{CSML}) \emph{of $L$ generated by $R$}.  The set of coincidence isometries of $L$ is written as $OC(L)$ and is equal to $OC(\G)$ \cite{LZxx}.  The 
		\emph{coincidence index of $R$ with respect to $L$}, denoted by $\S_{L}(R)$, is the ratio of the density of points in $L$ to the density of points in $L(R)$.  The CSMLs and
		coincidence indices of a multilattice may be obtained using the following result from~\cite{LZxx}.
		
		\begin{thm}\label{CSMLofL}
			Let $L=\bigcup_{k=0}^{m-1} (x_k+\G)$ be a multilattice generated by the lattice $\G$ and $R \in OC(L)$.
			\begin{enumerate}[i.]	
				\item The intersection $(x_k+\G)\cap R(x_j+\G)$ contains a cosublattice of $x_k+\G$ if and only if $Rx_j-x_k\in\G+R\G$. 
				
				\item Define 
				\begin{equation}\label{sigma}
					\sigma_L(R)=\set{(x_j,x_k):Rx_j-x_k\in\G+R\G}.
				\end{equation}
				Then $\S_{L}(R)=(m/|\sigma_L(R)|\,)\S_{\G}(R)$.
				
				\item For every $(x_j,x_k)\in\sigma_L(R)$ there exists an $\ell_{j,k}\in\G$ such that $Rx_j-x_k\in\ell_{j,k}+R\G$. Thus, 
				\[L(R)=\ts\bigcup_{(x_j,x_k)\in\sigma_L(R)}[(x_k+\ell_{j,k})+\G(R)].\]
			\end{enumerate}
		\end{thm}

	\section{Coincidences of the Hexagonal Lattice}
		Let $\xi=\exp(2\pi i/3)$.  We consider the hexagonal or triangular lattice \[\G=\Z[\xi]=\set{m+n\xi\mid m,n\in\Z},\] also known as the set of Eisenstein integers or the ring of
		integers of the cyclotomic field $\Q(\xi)$.  If $\gamma=m+n\xi\in\Z[\xi]$, then the conjugate of $\gamma$ is $\bar{\gamma}=(m-n)-n\xi$ and the number theoretic norm of $\gamma$ is 
		$N(\gamma)=\gamma\bar{\gamma}={(m-n)^2+mn}$. The ring $\Z[\xi]$ is a Euclidean domain under the given norm~\cite{HW08}, and hence also a principal ideal domain.  Its units form 
		the cyclic group $\set{1,-\xi^2,\xi,-1,\xi^2,-\xi}=\gen{-\xi^2}\cong C_6$ and they correspond to rotational point symmetries of $\G$. 
		
		The coincidence problem for the hexagonal lattice has been solved in~\cite{PBR96} (see also~\cite{BG06}).  Any coincidence rotation~$R$ of $\G$ by an angle of $\varphi$ corresponds
		to multiplication by 
			\begin{equation}\label{eiphi}
				\exp(i\varphi)=\vep\cdot\prod_{p \equiv 1(3)}\bigg(\frac{\;\omega_p\;}{\overline{\omega_p}}\bigg)^{t_p}
			\end{equation}
		where the product runs over all rational primes $p\equiv 1\imod 3$, $t_p\in\Z$ for which only finitely many of them are nonzero, $\vep$ is a unit in $\mathbbm{Z}[\xi]$, and $p$ is 
		split into its prime factors $\omega_p$ and its conjugate $\overline{\omega_p}$, i.e., $p=\omega_p\overline{\omega_p}$.  We define the numerator $z$ of $R$ by
		\begin{equation}\label{numerator}
			z=\prod_{p\equiv 1(3),t_p>0}{\omega_p^{t_p}}\cdot\prod_{p\equiv 1(3),t_p<0}{(\overline{\omega_p}\,)}^{-t_p}.
		\end{equation}
		Then the CSL obtained from $R$ is $\G(R)=z\mathbbm{Z}[\xi]=(z)$ generated by $z$ and $\vep z$.  In addition, the coincidence index of $R$ is the norm of $z$ in $\Z[\xi]$ given by
		\[\S_{\G}(R)=N(z)=\prod_{p\equiv 1(3)}{p^{|t_p|}}.\]
		
		Hence, the group of coincidence rotations of $\G$ has the structure $SOC(\G)\cong C_6\times\Z^{(\aleph_0)}$, where $\Z^{(\aleph_0)}$ denotes the direct sum of countably many
		infinite cyclic groups, each with a generator of the form $\omega_p/\overline{\omega_p}$ as in~\eqref{eiphi}.  Since every coincidence reflection $T$ can be expressed
		as $T=R\cdot T_{r}$, where $R\in SOC(\G)$ and $T_r$ is the reflection along the real axis (corresponding to complex conjugation), it follows that $OC(\G)=SOC(\G)\rtimes\gen{T_r}$.
			
		Denote by $f_{\Z[\xi]}(m)$ the number of CSLs of $\G$ of index $m$.  Since $\Z[\xi]$ is a unique factorization domain, $f_{\Z[\xi]}$ is a multiplicative arithmetic function, i.e., 
		$f_{\Z[\xi]}(mn) = f_{\Z[\xi]}(m)f_{\Z[\xi]}(n)$ whenever $m$ and $n$ are relatively prime.  The Dirichlet series generating function for $f_{\Z[\xi]}$ is given by 
		\[\begin{aligned}
			\Phi_{\Z[\xi]}(s)&=\sum_{m=1}^{\infty}{\frac{f_{\mathbbm{Z}[\xi]}(m)}{m^s}}=\prod_{p\equiv 1 (3)}\frac{1+p^{-s}}{1-p^{-s}}\\
			&=1+\tfrac{2}{7^s}+\tfrac{2}{13^s}+\tfrac{2}{19^s}+\tfrac{2}{31^s}+\tfrac{2}{37^s}+\tfrac{2}{43^s}+\cdots\;.
		\end{aligned}\]
		If $\hat{f}_{\Z[\xi]}(m)$ counts the number of coincidence rotations of $\G$ yielding CSLs of index $m$, then $\hat{f}_{\Z[\xi]}=6f_{\Z[\xi]}$.
		
	\section{Coincidences of a Shifted Hexagonal Lattice}
		We obtain the solution to the coincidence problem for a shifted hexagonal lattice using methods and techniques similar to the ones used for a shifted square lattice 
		(see~\cite{LZ10,LZxx}).  Details of the following results can be found in~\cite{G13}.  For the rest of this paper, we shall denote by $R_{z,\vep}$ and $T_{z,\vep}$ the
		coincidence rotation $R\in {SOC(x+\G)}$ and coincidence reflection $T\in OC(x+\G)$ associated with the numerator $z$ and unit $\vep$.  The following lemma allows us to identify 
		$OC(x+\G)$ for specific values of $x$.
		\begin{lemma}\label{ez-barz}
			Let $\G=\Z[\xi]$, $x=a+b\xi$ where $a,b\in\R$, $R=R_{z,\vep}\in SOC(\G),$ and $T=R\cdot T_r$.  Then
			\begin{enumerate}[i.]
				\item $R\in SOC(x+\G)$ if and only if $(\vep z-\bar{z})x\in\G$.
				
				\item $T \in OC(x+\G)$ if and only if $\vep z\bar{x}-\bar{z}x\in\G$.
			\end{enumerate}
		\end{lemma}
		
		The following results about the sets $OC(x+\G)$ and $SOC(x+\G)$ are consequences of Lemma~\ref{ez-barz}.
		\begin{thm}
			Let $\G=\Z[\xi]$ and $x=a+b\xi$, where $a,b\in\R$. 
			\begin{enumerate}[i.]
				\item $SOC(x+\G)$ is a subgroup of $SOC(\G)$.
				
				\item $OC(x+\G)$ is a subgroup of $OC(\G)$ if and only if for any coincidence reflections $T_1,T_2\in OC(x+\G)$, the coincidence rotation $T_2T_1\in SOC(x+\G)$.
				
				\item If $OC(x+\G)$ contains a reflection symmetry $T\in P(\G)$, then $OC(x+\G)$ is a group with $OC(x+\G) = SOC(x+\G)\rtimes\gen{T}$.
			\end{enumerate}
		\end{thm}
		
		\begin{cor}\label{varepsilon}
			Let $\G=\Z[\xi]$. If one of the following conditions on $x=a+b\xi$ holds: ${a\in\Z}, b\in\Z, a-b\in\Z, b-2a\in\Z, a+b\in\Z$ or $a-2b\in\Z$, then 
			\[OC(x+\G)=SOC(x+\G)\rtimes\gen{T_{1,\vep}}\] is a subgroup of $OC(\G)$ where
			\[\vep=\begin{cases}
				1,& \text{if } b\in\Z\\
				\xi,& \text{if } a-b\in\Z\\
				\xi^2, & \text{if } a\in\Z\\
				-1,& \text{if } -2a+b\in\Z\\				
				-\xi,& \text{if } a+b\in\Z\\				
				-\xi^2, & \text{if } a-2b\in\Z. 
			\end{cases}\]
		\end{cor}

		One uses the following lemma to distinguish which Eisenstein integers are possible numerators of coincidence rotations of $\G$.
		\begin{lemma}\label{mndiv}
			Let $\G=\Z[\xi]$ and $z=m+n\xi\in\Z[\xi]$ such that $\gcd(m,n)=1$. Then the following are equivalent: {\rm (i)} $z$ is a numerator of some $R\in SOC(\G)$,
			{\rm (ii)} $3\nmid N(z)$, and {\rm (iii)} $3\nmid (m+n)$.
		\end{lemma}

		The next theorem gives the solution of the coincidence problem for a shifted hexagonal lattice whenever one of the components of the shift vector is irrational.
		\begin{thm}
			Let $\G=\Z[\xi]$ and $x=a+b\xi$ for $a,b\in\R$. If $a$ or $b$ is irrational then $OC(x+\G)$ is a group having at most two elements. In particular, if
			\begin{enumerate}[i.]
				\item $a$ is irrational and $b$ is rational then
				\[OC(x+\G)=\begin{cases}
					\gen{T_r}, & \text{if } b\in\Z\\
					\set{\mathbbm{1}}, & \text{otherwise}.
				\end{cases}\]
				
				\item $a$ is rational and $b$ is irrational then 
				\[OC(x+\G)=\begin{cases}
					\gen{T_{1,\xi^2}}, & \text{if } a\in\Z\\
					\set{\mathbbm{1}}, & \text{otherwise}.
				\end{cases}\]
				
				\item both $a$ and $b$ are irrational, and
				\begin{enumerate}[\emph{(}a\emph{)}]
					\item 1, $a$, and $b$ are rationally independent then $OC(x+\G)=\set{\mathbbm{1}}$.
					
					\item $a=(p_1/q_1)+(p_2/q_2)b$ where $p_j, q_j\in\Z$, and $\gcd(p_j,q_j)=1$ for $j\in\set{1,2}$ with
					\begin{enumerate}[\emph{(}1\emph{)}]
						\item $p_2q_2\equiv 0,1\imod 3$, then 							
						\[OC(x+\G)=\begin{cases}
							\gen{T_{p_2 +q_2\xi,1}}, & \text{if } q_1\mid q_2\\
							\set{\mathbbm{1}}, &\text{otherwise}.
						\end{cases}\]
						
						\item $p_2q_2\equiv 2\imod 3$, then 
						\[OC(x+\G)=\begin{cases}
							\gen{T_{(2q_2-p_2)/3+(q_2-2p_2)\xi/3,-1}}, & \text{if } q_1\mid q_2\\
							\set{\mathbbm{1}}, &\text{otherwise}.
						\end{cases}\]
					\end{enumerate}
				\end{enumerate}
			\end{enumerate}
		\end{thm}
		
		Thus, it now remains to compute for $OC(x+\G)$ whenever both components of~$x$ are rational, i.e, whenever $x\in\Q(\xi)$.  Given $x\in\Q(\xi)$, we write $x=p/q$ where 
		$p,q\in\Z[\xi]$, and $p,q$ are relatively prime (in $\Z[\xi]$).  It turns out that $SOC(x+\G)$ ultimately depends on the value of the denominator $q$ of $x$.

		\begin{lemma}\label{qdiv}
			Let $x=p/q$ where $p,q\in\G=\Z[\xi]$ with $\gcd(p,q)=1$. Then $R_{z,\vep}\in SOC(x+\G)$ if and only if $q\mid(\vep z-\bar{z})$.
		\end{lemma}

		The succeeding theorem gives a further sufficient condition for $OC(x+\G)$ to form a group.
		
		\begin{thm}\label{split}
			Let $x=p/q$ where $p,q\in\G=\Z[\xi]$, and $p,q$ are relatively prime. If none of the prime factors $\pi$ of $N(q)$ satisfies $\pi\equiv 1\imod{3}$, then $OC(x+\G)$ is a group.
		\end{thm}
		
		We end this section by giving an example.  Recall that it suffices to compute $OC(x+\G)$ for values of $x$ in some fundamental domain $D$ of the symmetry group of $\G$.  A possible
		fundamental domain is $D=\set{a+b\xi \mid a,b\in\R,0\leq 4b\leq 2a\leq b+1}$ shown in Fig.~\ref{funddom}. 
		
		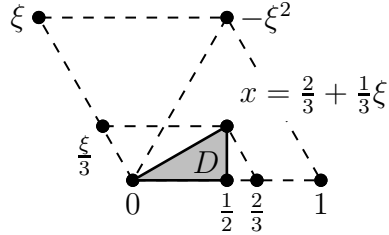
\begin{figure}[!ht]
			\psset{xunit=2.5cm,yunit=2.5cm}
			\centering
			\begin{pspicture*}(-0.7,-0.5)(1.5,1)
				\psdots[dotstyle=o,dotsize=5pt,fillcolor=black](0,0)	
				\psdots[dotstyle=o,dotsize=5pt,fillcolor=black](1,0)
				\psdots[dotstyle=o,dotsize=5pt,fillcolor=black](0.5,0.866)
				\psdots[dotstyle=o,dotsize=5pt,fillcolor=black](-0.5,0.866)
				\pspolygon[linewidth=1pt,fillstyle=solid,fillcolor=lightgray](0,0)(0.5,0)(0.5,0.288)
				\psdots[dotstyle=o,dotsize=5pt,fillcolor=black](0.5,0)
				\psdots[dotstyle=o,dotsize=5pt,fillcolor=black](0.5,0.288)
				\psdots[dotstyle=o,dotsize=5pt,fillcolor=black](0.66,0)
				\psdots[dotstyle=o,dotsize=5pt,fillcolor=black](-0.16,0.288)
				\psline[linewidth=0.75pt,linestyle=dashed,dash=4pt 4pt](0,0)(-0.5,0.866)(0.5,0.866)(1,0)(0,0)(0.5,0.866)
				\psline[linewidth=0.75pt,linestyle=dashed,dash=4pt 4pt](0.66,0)(0.5,0.288)(-0.16,0.288)
				\uput[d](0,0){\large$0$}
				\uput[d](0.5,0){\large$\frac{1}{2}$}
				\uput[d](0.66,0){\large$\frac{2}{3}$}
				\uput[d](1,0){\large$1$}
				\uput[dl](-0.16,0.288){\large$\frac{\xi}{3}$}
				\uput*{7pt}[ur](0.5,0.288){\large$x=\frac{2}{3}+\frac{1}{3}\xi$}
				\uput[l](-0.5,0.866){\large$\xi$}
				\uput[r](0.5,0.866){\large$-\xi^2$}
				\uput[r](0.25,0.1){\large{\bf$D$}}
			\end{pspicture*}
			\caption{\label{funddom} A fundamental domain $D$ for the symmetry group of $\G=\Z[\xi]$}
		\end{figure}
		
		\begin{ex}
			Let $x=(2+\xi)/3=1/(1-\xi)$ (see Fig.~\ref{funddom}).  The denominator of $x$ is $q=1-\xi$.  Let $z=m+n\xi$ be a numerator of some $R\in SOC(\G)$.  It follows from 
			Lemma~\ref{mndiv} and 
			\[\vep z-\bar{z}=\begin{cases}
				n\xi(1-\xi), &\text{if }\vep=1\\
				-m(1-\xi), &\text{if }\vep=\xi\\
				(m-n)\xi^2(1-\xi), &\text{if }\vep=\xi^2\\
				-2m+n, &\text{if }\vep=-1\\
				(m-2n)\xi^2, &\text{if }\vep=-\xi \\				
				(m+n)\xi, & \text{if }\vep=-\xi^2
			\end{cases}\] 
			that $q\mid(\vep z-\bar{z})$ if and only if $\vep\in\set{1,\xi,\xi^2}$.  Lemma~\ref{qdiv} now implies that 
			\[SOC(x+\G)=\set{R_{z,\vep}\in SOC(\G)\mid \vep\in \set{1,\xi,\xi^2}}\cong C_3\times\Z^{(\aleph_0)}.\]  
			Moreover, $OC(x+\G)=SOC(x+\G)\rtimes\gen{T_{1, -\xi^2}}$ and is a subgroup of $OC(\G)$ of index 2, by Corollary~\ref{varepsilon}. 
			If $f_{x+\G}(m)$ and $\hat{f}_{x+\G}(m)$ denote the number of CSLs and coincidence rotations of $x+\G$ of index $m$, respectively, then 
			$f_{x+\G}=f_{\Z[\xi]}$ while $\hat{f}_{x+\G}=3f_{x+\G}=(1/2)\hat{f}_{\Z[\xi]}$.
			Note that this result agrees with the discussion in the appendix of~\cite{PBR96} where the same conclusion was reached by applying a certain similarity transformation on the 
			hexagonal lattice.
		\end{ex}
			
	\section{Coincidences of the Hexagonal Packing}
		We now fix $x=1/(1-\xi)$.  Fig.~\ref{hexpack} illustrates that the hexagonal packing or honeycomb lattice can be viewed as the multilattice $L=\G\cup(x+\G)$.
		\begin{figure}[!ht]
			\psset{xunit=1.75cm,yunit=1.75cm}
			\centering\begin{pspicture*}(-0.3,-0.3)(4.3,3)
				\multido{\n=0+1.75}{4}{
					\put(\n,0){\pspolygon[linewidth=1.25pt,linestyle=dashed,dash=4pt 4pt,linecolor=gray](0,0)(0.5,0.866)(1,0)}
					\put(\n,0){\pspolygon[linewidth=1.25pt,linestyle=dashed,dash=4pt 4pt,linecolor=gray](0.5,0.866)(0,1.728)(1,1.728)}
					\put(\n,0){\pspolygon[linewidth=1.25pt,linestyle=dashed,dash=4pt 4pt,linecolor=gray](0,1.728)(0.5,2.594)(1,1.728)}
				}
				\psline[linewidth=1.25pt,linestyle=dashed,dash=4pt 4pt,linecolor=gray](0.5,0.866)(3.5,0.866)
				\psline[linewidth=1.25pt,linestyle=dashed,dash=4pt 4pt,linecolor=gray](0.5,2.594)(3.5,2.594)
				\multido{\n=0+1.75}{4}{
					\put(\n,0){\psline[linewidth=1.25pt,linecolor=black](0,0)(0.5,0.288)(0.5,0.866)(0,1.15)}
					\put(\n,0){\psline[linewidth=1.25pt,linecolor=black](1,1.15)(0.5,0.866)(0.5,0.288)(1,0)}
					\put(\n,0){\psline[linewidth=1.25pt,linecolor=black](0,1.15)(0,1.728)(0.5,2.016)(1,1.728)(1,1.15)}
				}
				\multido{\n=0+1.75}{3}{
					\put(\n,0){\psline[linewidth=1.25pt,linecolor=black](0.5,2.016)(0.5,2.594)(1,2.878)(1.5,2.594)(1.5,2.016)}
				}
				\multido{\n=0+1}{5}{
					\psdots[dotstyle=o,dotsize=5pt,fillcolor=gray](\n,0)
					\psdots[dotstyle=o,dotsize=5pt,fillcolor=black](\n,1.15)
					\psdots[dotstyle=o,dotsize=5pt,fillcolor=gray](\n,1.728)
				}
				\multido{\n=0.5+1.0}{4}{
					\psdots[dotstyle=o,dotsize=5pt,fillcolor=black](\n,0.288)	
					\psdots[dotstyle=o,dotsize=5pt,fillcolor=gray](\n,0.866)
					\psdots[dotstyle=o,dotsize=5pt,fillcolor=black](\n,2.016)
					\psdots[dotstyle=o,dotsize=5pt,fillcolor=gray](\n,2.594)
				}
				\multido{\n=1+1}{3}{
					\psdots[dotstyle=o,dotsize=5pt,fillcolor=black](\n,2.878)
				}
				\uput[d](0,0){\large$0$}
				\uput[d](1,0){\large$1$}
				\uput*{6pt}[37](0.5,0.288){$\frac{1}{1-\xi}$}
			\end{pspicture*}
			\caption{\label{hexpack} Hexagonal packing as the multilattice $L=\G\cup(x+\G)$, where $\G=\Z[i]$ and $x=1/(1-\xi)$}
		\end{figure}
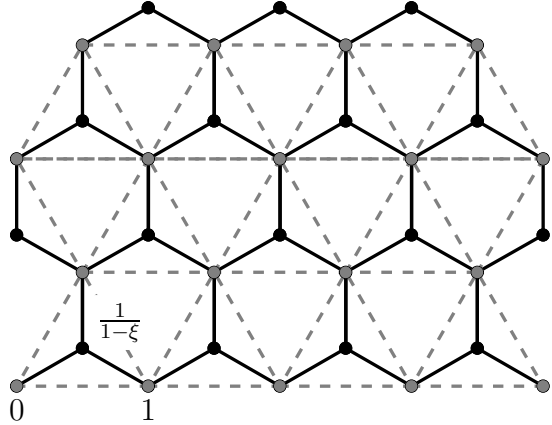
		
		Since $OC(L)=OC(\G)$, we only need to find $\S_L(R)$ and the CSML of~$L$ for an $R\in OC(\G)$.
		
		\begin{prop}\label{SigmaL}
			Let $\G=\Z[\xi]$, $L=\G\cup (x+\G)$ with $x=1/(1-\xi)$, and $R \in OC(\G)$.  
			\begin{enumerate}[i.]
				\item If $R\in OC(x+\G)$ then $\S_L(R)=\S_{\G}(R)$.
				
				\item If $R\notin OC(x+\G)$ then $\S_L(R)=2\cdot \S_{\G}(R)$.
			\end{enumerate}
		\end{prop}
		\begin{proof}
			To find $\S_L(R)$ for a given $R\in OC(\G)$, we examine which of the four possible pairs, $(0,0),(0,x),(x,0),(x,x)$, belong to $\sigma_L(R)$ from ~\eqref{sigma}.  Observe that
			the denominator $1-\xi$ of $x$ is a factor of the ramified prime $3$ in $\Z[\xi]$.
			
			If $R\in OC(x+\G)$, then one obtains (see~\cite{AGLxx} for details) that $\sigma_L(R)=\set{(0,0),(x,x)}$ and so $\S_L(R)=(2/2)\S_{\G}(R)=\S_{\G}(R)$.  On the other hand, if 
			$R\notin OC(x+\G)$ then $\sigma_L(R)=\set{(0,0)}$.  Hence, $\S_L(R)=(2/1)\S_{\G}(R)=2\cdot\S_{\G}(R)$.
		\end{proof}

		It follows from Theorem~\ref{CSMLofL} and the proof of Proposition~\ref{SigmaL} that the CSML of $L$ obtained from $R\in OC(\G)$ is $L(R)=\G(R)\cup [(x+\G)\cap R(x+\G)]$ if 
		$R\in OC(x+\G)$, and $L(R)=\G(R)$ otherwise.
		
	\section{Coincidences of a Shifted Hexagonal Packing}
		In the previous section, we established that the origin has only $D_3$ site symmetry with respect to the hexagonal packing $L=\G\cup (x+\G)$.  To obtain a maximal site symmetry of $
		D_6$, it is necessary to rotate $L$ about a different center.  Taking this center as $-x$, observe that rotating $L$ about $-x$ is equivalent to rotating $x+L$ about the origin. 
		This motivates us to investigate further the linear coincidences of a \emph{shifted multilattice $x+L$}.  The following theorem~\cite{AGLxx} tells us how to determine $OC(x+L)$,
		and gives formulas for $\S_{x+L}(R)$ and the CSML of $x+L$ for an $R\in OC(x+L)$.

		\begin{thm}\label{CSMLx+L}
			Let $L=\bigcup_{k=0}^{m-1}(x_k+\G)$ be a multilattice generated by the lattice $\G$ and
			$\sigma_{x+L}(R)\vcentcolon=\set{(x+x_j,x+x_k)\mid R(x+x_j)-(x+x_k)\in\G+R\G,0\leq j,k\leq m-1}$.
			\begin{enumerate}[i.]
				\item Then $OC(x+L)=\set{R\in OC(\G)\mid \sigma_{x+L}(R)\neq\varnothing}$.
				
				\item The coincidence index of $R\in OC(x+L)$ with respect to $x+L$ is \[\S_{x+L}(R)=(m/|\sigma_{x+L}(R)|)\S_{\G}(R).\]
		
				\item For every $(x+x_j,x+x_k)\in\sigma_{x+L}(R)$, there exists an $\ell_{j,k}\in\G$ such that $R(x+x_j)-(x+x_k)\in\ell_{j,k}+R\G$.  Thus, the CSML of $x+L$ is given by 
				\[(x+L)\cap R(x+L)=\ts\bigcup_{(x_j,x_k)\in\sigma_{x+L}(R)}[(x+x_k+\ell_{j,k})+\G(R)].\]
			\end{enumerate}
		\end{thm}

		Since $x\in -2x+\G$, the hexagonal packing may be written as $L=\G\cup(x+\G)=\G\cup(-2x+\G)$.  Observe that the centers of the hexagons in $L$ are points with maximal site symmetry
		$D_6$.  One of these centers is $-x$, so we look at coincidences of the shifted multilattice $x+L=(x+\G)\cup(-x+\G)$.

		\begin{prop}\label{Sigmax+Lhexagonal}
			Let $L=\G\cup(-2x+\G)$ where $\G=\Z[\xi]$ and $x=1/(1-\xi)$.  Then $OC(x+L)=OC(\G)$ and for all $R\in OC(x+L)$, $\S_{x+L}(R)=\S_{\G}(R)$.
		\end{prop}
		\begin{proof}
			Analogous to the proof of Proposition~\ref{SigmaL}, we need to find which of the four pairs $(x,x)$, $(x,-x)$, $(-x,x)$, $(-x,-x)$ belong to $\sigma_{x+L}(R)$ for an 
			$R\in OC(\G)$.  The following are due largely to Theorem~\ref{CSMLofL}, and full details may be found in [12].

			If $R\in OC(x+\G)$, $\sigma_{x+L}(R)=\set{(x,x),(-x,-x)}$, and therefore $\S_{x+L}(R)=(2/2)\S_{\G}(R)=\S_{\G}(R)$. Otherwise, $\sigma_{x+L}(R)=\set{(x,-x),(-x,x)}$, which
			yields $\S_{x+L}(R)=(2/2)\S_{\G}(R)=\S_{\G}(R)$.  In both cases, the conclusion now follows.
		\end{proof}

		Proposition~\ref{Sigmax+Lhexagonal} also appears in the Appendix of~\cite{PBR96}.  Nonetheless, we gain from this different approach the following information about the points of
		intersection by applying Theorem~\ref{CSMLx+L}.  If $R\in OC(x+\G)$, then the CSML of $x+L$ generated by $R$ is the union of the CSLs of $x+\G$ and $(-x+\G)$
		obtained from $R$.  On the other hand, the CSML generated by $R\notin OC(x+\G)$ is obtained from the intersection of $R(-x+\G)$ with $x+\G$ and the intersection
		of $R(x+\G)$ with $-x+\G$.	

	\section*{Acknowledgement}
		M.J.C.~Loquias would like to thank the Office of the Chancellor of the University of the Philippines Diliman, through the Office of the Vice Chancellor for Research and Development,
		for funding support through the Ph.D.~Incentive Awards.
	
	\bibliographystyle{amsplain}
	\bibliography{bibliog}
\end{document}